\documentclass[12pt, reqno]{amsart}
\usepackage{amsmath, amsfonts, amsbsy, amsthm, amscd, graphicx}
\usepackage{amssymb, latexsym, comment}
\usepackage{bm}
\usepackage{enumerate}
\usepackage[svgnames,dvipsnames]{xcolor}
\usepackage[colorlinks, citecolor=Green,linkcolor=FireBrick,linktocpage,unicode]{hyperref}

\theoremstyle{definition}
\numberwithin{equation}{section}
\newtheorem{theo}{Theorem}[section]
\newtheorem{defi}[theo]{Definition}
\newtheorem{lemm}[theo]{Lemma}
\newtheorem{rema}[theo]{Remark}
\newtheorem{corr}[theo]{Corollary}
\newtheorem{exam}[theo]{Example}
\newtheorem{prop}[theo]{Proposition}

\begin{document}

\title{Remark on Laplacians and Riemannian Submersions with Totally Geodesic Fibers}
\author{Kazumasa Narita}
\thanks{National Institute of Technology, Yonago College, Yonago, Tottori, 683-8502, Japan, narita@yonago-k.ac.jp, ORCID iD: 0009-0002-0833-5721}
\date{}

\maketitle

\begin{abstract}
Given a Riemannian submersion $(M,g) \to (B,j)$ each of whose fibers is connected and totally geodesic, we consider a certain 1-parameter family of Riemannian metrics $(g_{t})_{t > 0}$ on $M$, which is called the canonical variation. Let $\lambda_{1}(g_{t})$ be the first positive eigenvalue of the Laplace--Beltrami operator $\Delta^{M}_{g_{t}}$ and $\mbox{Vol}(M,g_{t})$ the volume of $(M, g_{t})$. In 1982, B\'{e}rard-Bergery and Bourguignon showed that the scale-invariant quantity $\lambda_{1}(g_{t})\mbox{Vol}(M,g_{t})^{2/\mbox{dim}M}$ goes to $0$ with $t$. In this paper, we show that if each fiber is Einstein and $(M,g)$ satisfies a certain condition about its Ricci curvature, then bounds for $\lambda_{1}(g_{t})$ can be obtained. In particular this implies that $\lambda_{1}(g_{t})\mbox{Vol}(M,g_{t})^{2/\mbox{dim}M}$ goes to $\infty$ with $t$. Moreover, using these bounds, we consider stability of $g_{t}$ as a critical point of the Yamabe functional. 
\end{abstract}

\textbf{Keywords} Laplacian eigenvalue $\cdot$ Riemannian submersion $\cdot$ Einstein manifold $\cdot$ Yamabe functional

\textbf{Mathematics Subject Classification} 53C20, 53C25

\section{Introduction}
Let $M$ be a compact (connected) manifold of dimension $n$ (without boundary). Given a Riemannian metric $g$ on $M$, the volume $\mbox{Vol}(M,g)$ and the Laplace--Beltrami operator $\Delta_{g}^{M}$ are defined. When there is no room for confusion, we sometimes abbreviate $\Delta_{g}^{M}$ as $\Delta^{M}$. Let $0 = \lambda_{0}(g) < \lambda_{1}(g) < \lambda_{2}(g) < \cdots < \lambda_{l}(g) < \cdots$ be the eigenvalues of $\Delta^{M}$. The quantity $\lambda_{l}(g)\mbox{Vol}(M,g)^{2/n}$ is invariant under scaling of the metric $g$. In 1970, Hersch \cite{Hersch} proved that on a $2$-dimensional sphere $S^{2}$, the quantity $\lambda_{1}(g)\mbox{Area}(S^{2},g)$ is maximized precisely when $g$ is the round metric. In 1973, inspired by this work, Berger \cite{Berger} posed the question whether 
\begin{equation*}
\Lambda_{1}(M) := \sup_{g} \lambda_{1}(g)\mbox{Vol}(M,g)^{2/n}
\end{equation*}
is finite for a compact manifold $M$ of dimension $n$. This question is equivalent to asking whether the supremum of the functional $\lambda_{1}$ is finite over all the Riemannian metrics with fixed volume. For a closed surface $M$, $\Lambda_{1}(M)$ is bounded by a constant depending on the genus of $M$ \cite{YY, Karpukhin}. In contrast, in 1979, Urakawa \cite{Urakawa} considered a 1-parameter family of Riemannian metrics $(h_{t})_{t > 0}$ on  SU(2) $\cong S^{3}$ with fixed volume and obtained an explicit formula for  $\lambda_{1}(h_{t})$. This in particular implies that the scale-invariant quantity $\lambda_{1}(h_{t})\mbox{Vol}(S^{3}, h_{t})^{2/3}$ goes to $\infty$ as $t$ goes to $\infty$, and to $0$ as $t$ goes to $0$. This was the first negative answer to Berger's aforementioned question. Soon after this Urakawa's work, Tanno \cite{Tanno} considered the Hopf fibration $S^{1} \to S^{2n+1} \to \mathbf{C}P^{n}$ and extended the Urakawa's result to any odd dimensional sphere. In 1994, Colbois and Dodziuk \cite{CD} showed that on any compact manifold $M$ of dimension $n \geq 3$, there exists a 1-parameter family of Riemannian metrics $(k_{t})_{t > 0}$ such that $\lambda_{1}(k_{t})\mbox{Vol}(M,k_{t})^{2/n}$ goes to $\infty$ with $t$. Hence the Berger's question was shown to be false for any compact manifold. We remark that in contrast to the metrics considered by Urakawa and Tanno, those constructed by Colbois and Dodziuk cannot be written explicitly.

This paper is largely inspired by B\'{e}rard-Bergery and Bourguignon's work \cite{BBB} in 1982. They considered a Riemannian submersion $\pi: (M, g) \to (B,j)$ each of whose fibers is connected and totally geodesic. They constructed a 1-parameter family of Riemannian metrics $(g_{t})_{t > 0}$ on $M$, which is called the \textit{canonical variation} (see the original paper \cite{BBB} or Section 2 of this paper for the precise definition). In this paper, we always assume that each fiber of a Riemannian submersion $\pi: (M, g) \to (B,j)$ is connected and totally geodesic and $(g_{t})_{t>0}$ is the canonical variation. Set
\begin{equation*}
\Lambda_{1}(M, t) := \lambda_{1}(g_{t})\mbox{Vol}(M,g_{t})^{2/n}.
\end{equation*}
B\'{e}rard-Bergery and Bourguignon \cite{BBB} showed that $\Lambda_{1}(M, t)$ goes to $0$ as $t$ goes to $0$. This can be seen as partial generalization of the aforementioned work due to Urakawa \cite{Urakawa} and Tanno \cite{Tanno}. This qualitative behavior of $\Lambda_{1}(M, t)$ is deduced from the squeeze theorem and the following quantitative result:

\begin{prop}[\cite{BBB}]
\label{Intro-prop}
\textit{Let $(M, g)$ and $(B, j)$ be compact Riemannian manifolds of dimension $n$ and $p$ $(n>p)$ respectively. Let $\pi: (M, g) \to (B,j)$ be a Riemannian submersion. Let $\beta_{1}$ be the smallest positive eigenvalue of $\Delta^{B}_{j}$. Then for any $t > 0$, one has 
\begin{equation*}
 \lambda_{1}(g_{t}) \leq \beta_{1}.
\end{equation*}
The  equality holds if and only if the space of pullback functions of $\beta_{1}$-eigenfunctions from $(B,j)$ to $(M, g_{t})$ is a nontrivial subspace of the $\lambda_{1}(g_{t})$-eigenfunction space.}
\end{prop}

 From the perspective of the above history about the Berger's question, it is natural to ask when $\Lambda_{1}(M, t)$ goes to $\infty$ with $t$. B\'{e}rard-Bergery and Bourguignon \cite{BBB} pointed out that if the Riemannian submersion $\pi$ satisfies the H\"{o}rmnader condition, then $\Lambda_{1}(M, t)$ goes to $\infty$ with $t$. (See Proposition \ref{Hormander} for the precise statement.) However, only such a condition gives no estimates for $\lambda_{1}(g_{t})$. In 2016, Baudoin and Kim \cite{BK} proved that if the sub-Riemannian structure is of H-type in addition and the horizontal Ricci curvature is positive, then a sharp lower bound for the first eigenvalue of the horizontal Laplacian can be found. From this bound, one can immediately obtain a lower bound for $\lambda_{1}(g_{t})$ and conclude that $\Lambda_{1}(M, t)$ goes to $\infty$ with $t$. In the main theorem of this paper, under the assumption only on Ricci curvatures, we give bounds for $\lambda_{1}(g_{t})$, namely: 
\begin{theo}
\label{MainThm}
\textit{Let $(M, g)$ and $(B, j)$ be compact Riemannian manifolds of dimension $n$ and $p$ $(n>p)$ respectively. Let $\pi: (M, g) \to (B,j)$ be a Riemannian submersion. Assume that there exists $\widetilde{c} > 0$ such that 
\begin{equation*}
\mbox{Ric}^{M} \geq \widetilde{c}g,
\end{equation*}
where $\mbox{Ric}^{M}$ is the Ricci tensor of $(M,g)$. In case $p \leq n-2$, assume also that there exists $0 \leq c < \widetilde{c}$ such that 
\begin{equation*}
\mbox{Ric}^{F_{y}} = c(\iota^{*}g)
\end{equation*}
for any $y \in B$, where $F_{y}$ is a fiber of $y$ and $\iota:F_{y} \hookrightarrow M$ is the inclusion map. Then for any $t \geq 1$, one has 
\begin{equation}
\label{main-estimate}
\frac{\widetilde{c}-c}{n+1} + t^{-2}\left( \frac{n^{2}+1}{n^{2}-1}\widetilde{c}+\frac{c}{n+1} \right) \leq \lambda_{1}(g_{t}) 
\end{equation}
The equality holds if and only if $t =1$ and $(M,g)$ is isometric to the odd dimensional round sphere $S^{2m+1}(\sqrt{2m/\widetilde{c}})$. }
\end{theo}
The curvature assumptions of this theorem are natural in a sense that all the Riemannian submersions from compact rank one symmetric spaces satisfy them. By Proposition \ref{Intro-prop} and Theorem \ref{MainThm}, one can conclude the following asymptotic behavior of $\Lambda_{1}(M, t)$:
\begin{corr}
\label{Intro-corr}
\textit{ Under the assumption of Theorem \ref{MainThm}, one has 
\begin{equation*}
\Lambda_{1}(M, t) = O(t^{2(n-p)/n}) \quad ( t \to \infty).
\end{equation*}
}
\end{corr}

Even after Berger's aforementioned question was settled by Colbois and Dodziuk \cite{CD}, the study of $\lambda_{1}(g_{t})$ has been of interest since it is related to stability of critical points of the Yamabe functional. Yamabe functional $Y_{C}$ is the normalized total scalar curvature functional restricted to the conformal class $C$. A metric $g \in C$ is a critical point of $Y_{C}$ if and only if $g$ has constant scalar curvature. Given a Riemannian submersion $(M,g) \to (B,j)$ such that $g$ is Einstein, then $g_{t}$ is a critical point of the Yamabe functional $Y_{[g_{t}]}$ for any $t >0$. However, whether it is a stable critical point depends on its scalar curvature and $\lambda_{1}(g_{t})$.  In 2022, Bettiol, Lauret and Piccione \cite{BLP} completed computation of $\lambda_{1}(g_{t})$ for all the Riemannian submersions from compact rank one symmetric spaces and determined the range of $t$ where $g_{t}$ is stable for each of these submersions. In this paper, we consider a more general setting and as a corollary of Theorem \ref{MainThm}, we deduce a sufficient condition for $g_{t}$ to be stable. Before stating this result, we recall that for a Riemannian submersion $(M,g) \to (B,j)$, O'Neill \cite{ONeill} introduced the $(2,1)$ tensor field $A$ on $M$ whose values on vector fields $E_{1}$ and $E_{2}$ are given by
\begin{equation*}
A_{E_{1}}E_{2} = \mathcal{H}\nabla_{\mathcal{H}E_{1}}\mathcal{V}E_{2} + \mathcal{V}\nabla_{\mathcal{H}E_{1}}\mathcal{H}E_{2},
\end{equation*}
where $\mathcal{H}$ and $\mathcal{V}$ denote the horizontal and vertical projection respectively. It is known that if $g$ is Einstein, then $|A|^{2}$ is constant on $M$ (see \cite[p.250]{Besse}).
\begin{theo}[Theorem \ref{stability}]
\label{Intro}
\textit{ Let $(M, g)$ and $(B, j)$ be connected compact Riemannian manifolds of dimension $n$ and $p$ $(n>p)$ respectively. Assume that $(M,g)$ is not isometric to the round sphere. Let $\pi: (M, g) \to (B,j)$ be a Riemannian submersion. Assume that there exists $\widetilde{c} > 0$ such that 
\begin{equation*}
\mbox{Ric}^{M} = \widetilde{c}g,
\end{equation*}
where $\mbox{Ric}^{M}$ is the Ricci tensor of $(M,g)$. Assume also that $|A|$ is strictly positive on $M$. In case $p \leq n-2$, assume also that there exists $0 \leq c < \widetilde{c}$ such that 
\begin{equation*}
\mbox{Ric}^{F_{y}} = c(\iota^{*}g)
\end{equation*}
for any $y \in B$. Set
\begin{equation*}
\Gamma := \frac{n^{2}+1}{n+1}\left( \widetilde{c}-c \right) +pc.
\end{equation*}
Then for any $t \geq \max\{1, \: \sqrt{\Gamma}/|A|\}$, $g_{t}$ is stable. }
\end{theo}

In Section 5, we will see that Theorem \ref{MainThm} and Theorem \ref{Intro} can be applied to many examples, in particular the twistor fibration of a quaternionic K\"{a}hler manifold of positive scalar curvature.  

\textbf{Organization.} In Section 2, we will review the work of B\'{e}rard-Bergery and Bourguignon \cite{BBB}. In Section 3, we will prove Theorem \ref{MainThm}. The essential part of its proof is Lemma \ref{key-lemma}. This lemma is of interest itself and proved by modification of the proof of the Lichnerowicz--Obata theorem. In Section 4, after a review of basic facts of the Yamabe problem, we will prove Theorem \ref{Intro} (Theorem \ref{stability}). In Section 5, we will see many examples. In particular, we will consider the canonical variation of the K\"{a}hler--Einstein metric on the twistor space of a quaternionic K\"{a}hler manifold of positive scalar curvature.

\section{Review of the work of B\'{e}rard-Bergery and Bourguignon}
First we briefly review B\'{e}rard-Bergery and Bourguignon's work \cite{BBB}. Let $(M, g)$ and $(B, j)$ be connected compact Riemannian manifolds of dimension $n$ and $p$ respectively. Let $\pi: (M, g) \to (B,j)$ be a Riemannian submersion. For $m \in M$, the fiber at $\pi(m) \in B$ is denoted by $F_{\pi(m)}$. Each fiber $F_{\pi(m)}$ is a $(n-p)$-dimensional submanifold of $M$. Let $\iota : F_{\pi_{m}} \hookrightarrow M$ be the inclusion map. Then the following Fubini-type formula holds: for any integrable function $f$ on $(M, g)$, one has
\begin{equation}
\label{Fubini}
\int_{M} f(x) \: d\mu_{g}(x) = \int_{B} \left[ \int_{F_{y}} f(z)d\mu_{\iota^{*}g}(z) \right] d\mu_{j}(y),
\end{equation}
which can be seen easily by letting $y= \pi(x)$ and integration by parts. 
In \cite{BBB}, B\'{e}rard-Bergery and Bourguignon introduced the \textit{vertical Laplacian} $\Delta_{v}$ and the \textit{horizontal Laplacian} $\Delta_{h}$ as follows: for a smooth function $f$ on $M$, $\Delta_{v} f$ and $\Delta_{h}f$ are given by
\begin{equation*}
(\Delta_{v} f) (m) := \Delta^{F_{\pi(m)}} (f \restriction_{F_{\pi(m)}}) (m) 
\end{equation*}
and
\begin{equation*}
\Delta_{h}f := \Delta^{M} f- \Delta_{v}f, 
\end{equation*}
where $\Delta^{F_{\pi(m)}}$ is the Laplacian on the Riemannian manifold $(F_{\pi(m)}, \iota^{*}g)$ and $\Delta^{M}$ is that on $(M,g)$. These operators are (formally) self-adjoint nonnegative operator on $L^{2}(M,g)$. As B\'{e}rard-Bergery and Bourguignon remarked, these operators are not elliptic and so the names of these operators might be misleading. The spectrum of $\Delta_{v}$ is discrete, but that of $\Delta_{h}$ is not necessarily so. Hermann \cite{Hermann} showed the following fundamental result:

\begin{lemm}[\cite{Hermann}]
\label{lemm-Hermann}
\textit{
Let $(M, g)$ and $(B, j)$ be connected complete Riemannian manifolds. Let $\pi: (M, g) \to (B,j)$ be a Riemannian submersion. Assume that each of the fibers equipped with the induced metric is totally geodesic in $(M, g)$. Then the fibers are isometric to each other. }
\end{lemm}

In this paper, we always assume connectedness of manifolds and that the fibers of a Riemannian submersion are totally geodesic. Using Hermann's result \cite{Hermann}, B\'{e}rard-Bergery and Bourguignon showed the following:

\begin{lemm}[\cite{BBB}]
\label{lemm-BBB}
\textit{
Let $(M, g)$ and $(B, j)$ be compact Riemannian manifolds of dimension $n$ and $p$ respectively. Let $\pi: (M, g) \to (B,j)$ be a Riemannian submersion. Let $\{\tilde{e}_{i} \}_{i=1}^{p}$ be a local orthonormal frame on $(B, j)$. For any $1 \leq i \leq p$, let $e_{i}$ be the horizontal lift of $\tilde{e}_{i}$. Let $\{e_{i}\}_{i=p+1}^{n}$ be local vertical vector fields such that $\{e_{i}\}_{i=1}^{n}$ is a local orthonormal frame on $(M,g)$. Then one has 
\begin{equation}
\label{local-v}
\Delta_{v} f= -\sum_{i=p+1}^{n} [e_{i}(e_{i}f)-(\nabla^{M}_{e_{i}}e_{i})f] 
\end{equation}
and
\begin{equation}
\label{local-h}
\Delta_{h} f= -\sum_{i=1}^{p} [e_{i}(e_{i}f)-(\nabla^{M}_{e_{i}}e_{i})f],
\end{equation}
where $\nabla^{M}$ is the Levi--Civita connection of $(M,g)$. Moreover, the operators $\Delta^{M}$, $\Delta_{v}$ and $\Delta_{h}$ commute with each other. }
\end{lemm}
B\'{e}rard-Bergery and Bourguignon considered  a 1-parameter family of Riemannian metrics associated with a Riemannian submersion, which is called the \textit{canonical variation}:

\begin{defi} Let $(M, g)$ and $(B, j)$ be Riemannian manifolds. Let $\pi: (M, g) \to (B,j)$ be a Riemannian submersion. For $m \in M$, the vertical (\textit{resp.} horizontal) space is denoted by $V_{m}M$ (\textit{resp.} $H_{m}M$). For each $t >0$, let $g_{t}$ be the unique Riemannian metric on $M$ satisfying the following three conditions:
\begin{enumerate}
\item $g_{t} \restriction_{V_{m}M \times H_{m}M} = 0$.
\item $g_{t} \restriction_{H_{m}M \times H_{m}M} = g \restriction_{H_{m}M \times H_{m}M}$.
\item $g_{t} \restriction_{V_{m}M \times V_{m}M} = t^{2}g \restriction_{V_{m}M \times V_{m}M}$.
\end{enumerate}
The 1-parameter family of Riemannian metrics $(g_{t})_{t>0}$ is called the \textit{canonical variation} of $g$ associated with $\pi: (M, g) \to (B,j)$. 
\end{defi}
In the above definition, one has $g_{1} =g$. Moreover, the Fubini-type formula $(\ref{Fubini})$ implies 
\begin{equation}
\label{volume}
\mbox{Vol}(M, g_{t}) = \mbox{Vol}(M,g)t^{n-p}.
\end{equation}
Also, it is easy to see that if $\pi: (M, g) \to (B,j)$ is a Riemannian submersion with totally geodesic fibers, then so is $\pi: (M, g_{t}) \to (B,j)$ for any $t >0$. If $\pi: (M, g) \to (B,j)$ is a Riemannian submersion with totally geodesic fibers, then one can write $\Delta^{M}_{g_{t}}$  only with the operators on $(M,g)$, namely:
\begin{equation}
\label{variation-Laplacian}
\Delta^{M}_{g_{t}} = t^{-2}\Delta_{v} + \Delta_{h} = t^{-2}\Delta^{M} + (1-t^{-2}) \Delta_{h}.
\end{equation}
As for the canonical variation $(g_{t})_{t > 0}$, set
\begin{equation*}
\Lambda_{1}(M, t) := \lambda_{1}(g_{t})\mbox{Vol}(M,g_{t})^{2/n}.
\end{equation*}
One can easily see that the pull-back of a $\Delta^{B}_{j}$-eigenfunction is an $\Delta^{M}_{g_{t}}$-eigenfunction on $(M, g_{t})$. Hence one concludes 
\begin{equation}
\label{upper-bound}
\lambda_{1}(g_{t}) \leq \beta_{1}
\end{equation}
for any $t >0$, where $\beta_{1}$ is the smallest positive eigenvalue of $\Delta^{B}_{j}$. In addition, Grama and Lima \cite{GL} recently pointed out that for any $0<t \leq 1$, the inequality 
\begin{equation}
\label{small-t}
\lambda_{1}(g) \leq \lambda_{1}(g_{t}) 
\end{equation}
holds. Combining $(\ref{volume})$, $(\ref{upper-bound})$ and $(\ref{small-t})$, one obtains the following result:
\begin{theo}[c.f. \cite{BBB}]
\label{BBB-theo}
\textit{
Let $(M, g)$ and $(B, j)$ be compact Riemannian manifolds of dimension $n$ and $p$ respectively. Let $\pi: (M, g) \to (B,j)$ be a Riemannian submersion. Then for any $0<t \leq 1$, one has
\begin{equation*}
\lambda_{1}(g) \leq \lambda_{1}(g_{t}) \leq \beta_{1}
\end{equation*}
and so
\begin{equation*}
\lambda_{1}(g)\mbox{Vol}(M,g)^{2/n}t^{2(n-p)/n} \leq \Lambda_{1}(M,t) \leq \beta_{1}\mbox{Vol}(M,g)^{2/n}t^{2(n-p)/n}. 
\end{equation*}
In particular, one has $\Lambda_{1}(M, t) =O(t^{2(n-p)/n}) \quad (t \to 0)$. 
}
\end{theo}
The fact that $\lim_{t\to 0} \Lambda_{1}(M, t) =0$ was first proved by B\'{e}rard-Bergery and Bourguignon \cite{BBB} by a different discussion. It is natural to ask when $\Lambda_{1}(M, t)$ goes to $\infty$ with $t$. B\'{e}rard-Bergery and Bourguignon observed the following:

\begin{prop}[\cite{BBB}]
\label{before-Hormander}
\textit{
Let $(M, g)$ and $(B, j)$ be compact Riemannian manifolds of dimension $n$ and $p$ respectively. Let $\pi: (M, g) \to (B,j)$ be a Riemannian submersion. Assume that $\Delta_{h} f=0$ implies $f$ constant and that $0$ is not an accumulation point of the spectrum of $\Delta_{h}$. Then $\Lambda_{1}(M, t)$ goes to $\infty$ as $t$ goes to $\infty$.}
\end{prop}
They also pointed out a sufficient condition for the assumption of Proposition \ref{before-Hormander} to hold.
\begin{prop}[\cite{BBB}]
\label{Hormander}
\textit{ Let $(M, g)$ and $(B, j)$ be compact Riemannian manifolds of dimension $n$ and $p$ respectively. Let $\pi: (M, g) \to (B,j)$ be a Riemannian submersion each of whose fibers is connected and totally geodesic. Assume that all the  basic vector fields and their iterated Lie brackets span $TM$. Then the assumption of Proposition \ref{before-Hormander} holds and so in particular $\Lambda_{1}(M, t)$ goes to $\infty$ as $t$ goes to $\infty$. }
\end{prop}
Recall that a vector field $X$ on $M$ is called \textit{basic} if $X$ is the horizontal lift of some vector field on $B$. The assumption of Proposition \ref{Hormander} is known as the H\"{o}rmander condition. However, one cannot obtain an explicit estimate for $\lambda_{1}(g_{t})$ from the H\"{o}rmander condition.

\section{Proof of Theorem \ref{MainThm}}
Before proving Theorem \ref{MainThm}, we remark that the assumption of Theorem \ref{MainThm} requires $n \geq 3$. In fact, if $n=2$, then the assumption and the Gauss--Bonnet theorem implies that $M$ is homeomorphic to $S^{2}$ or $\mathbf{R}P^{2}$. Since $B$ is a one-dimensional connected compact manifold, $B$ is homeomorphic to $S^{1}$. However, it is a standard topological fact that neither $S^{2}$ nor $\mathbf{R}P^{2}$ admit submersions to $S^{1}$. Hence we conclude $n \geq 3$.

The essential part of the proof of Theorem \ref{MainThm} is the following lemma:
\begin{lemm}
\label{key-lemma}
\textit{
Let $(M, g)$ and $(B, j)$ be compact Riemannian manifolds of dimension $n$ and $p$ respectively. Let $\pi: (M, g) \to (B,j)$ be a Riemannian submersion. Assume that there exists $\widetilde{c} > 0$ such that 
\begin{equation*}
\mbox{Ric}^{M} \geq \widetilde{c}g,
\end{equation*}
where $\mbox{Ric}^{M}$ is the Ricci tensor of $(M,g)$. In case $p \leq n-2$, assume also that there exists $0 \leq c < \widetilde{c}$ such that 
\begin{equation*}
\mbox{Ric}^{F_{y}} = c(\iota^{*}g)
\end{equation*}
for any $y \in B$, where $F_{y}$ is a fiber of $y$ and $\iota: F_{y} \hookrightarrow M$ is the inclusion map. Set
\begin{equation*}
\alpha_{k} := \left[(\lambda_{k}(g)-c)+\frac{\lambda_{k}(g)^{2}}{n(\lambda_{k}(g)-\widetilde{c})} \right] p
\end{equation*}
and 
\begin{equation*}
\beta_{k} :=  \frac{\widetilde{c}-c}{\lambda_{k}(g)-\widetilde{c} } \times \frac{\lambda_{k}(g)^{2}p}{n},
\end{equation*}
where in case $p=n-1$, substitute $c=0$. Define the quadratic function $Q_{k}(x)= (p+1)x^{2}-\alpha_{k}x+\beta_{k}$ on $[0, \infty)$. Let $f$ be a smooth function on $M$ such that $\int_{M}f^{2} d\mu_{g} =1$, $\Delta^{M}f= \lambda_{k}(g)f$, $\Delta_{h}f=af$. Then $a$ satisfies $a > \widetilde{c}-c$ or $Q_{k}(a) \leq 0$.
\\
Furthermore, $a$ satisfies $a \leq \widetilde{c}-c$ and $Q_{k}(a) = 0$ if and only if $(M,g)$ is isometric to the odd dimensional round sphere $S^{2m+1}(\sqrt{2m/\widetilde{c}})$ and $f$ is a linear combination of coordinate functions, namely, $k=1$.}
\end{lemm}

\begin{proof}
We assume $a \leq \widetilde{c}-c$ and prove $Q_{k}(a) \leq 0$. The Lichnerowicz--Obata theorem implies $\lambda_{1}(g) \geq n\widetilde{c}/(n-1)$. If $f$ is the pull-back of a function on $B$, then one has
\begin{equation*}
a = \lambda_{k}(g) \geq \lambda_{1}(g) \geq \frac{n}{n-1}\widetilde{c}>\widetilde{c}-c,
\end{equation*}
which contradicts the assumption. Hence $f$ is not a pull-back of a function on $B$. One has
\begin{equation*}
a= \int_{M} f\Delta_{h}f \: d\mu_{g}= \int_{M} f(\Delta^{M}- \Delta_{v})f \: d\mu_{g}.
\end{equation*}
Hence the Fubini-type formula $(\ref{Fubini})$ and Stokes' theorem imply
\begin{equation}
\label{a}
a = \int_{M} |\nabla^{M}f|^{2} d\mu_{g} - \int_{B}\left[  \int_{F_{y}} |\nabla^{F_{y}}(f\restriction_{F_{y}})|^{2} d\mu_{\iota^{*}g}(z) \right] d\mu_{j}(y).
\end{equation}
\underline{Case: $p \leq n-2$:} The Bochner formula says that 
\begin{equation*}
-\frac{1}{2} \Delta^{M} |\nabla^{M}f|^{2} = |\mbox{Hess}^{M}f|^{2} - g(\nabla^{M} f, \nabla^{M}(\Delta^{M}f))+ \mbox{Ric}^{M}(\nabla^{M}f, \nabla^{M}f)
\end{equation*}
and 
\begin{equation}
\label{comparison}
\begin{split}
-\frac{1}{2} \Delta^{F_{y}} |\nabla^{F_{y}}(f\restriction_{F_{y}})|^{2} &= |\mbox{Hess}^{F_{y}}(f\restriction_{F_{y}})|^{2} - g(\nabla^{F_{y}}( f\restriction_{F_{y}}), \nabla^{F_{y}}(\Delta_{v}f)) \\
&\quad + \mbox{Ric}^{F_{y}}(\nabla^{F_{y}}(f\restriction_{F_{y}}), \nabla^{F_{y}}(f\restriction_{F_{y}})) \\
\end{split}
\end{equation}
hold for any $ y \in B$. By the assumption, these equations imply
\begin{equation}
\label{Bochner-1}
-\frac{1}{2} \Delta^{M} |\nabla^{M}f|^{2} \geq |\mbox{Hess}^{M}f|^{2} - (\lambda_{k}(g)-\widetilde{c}) |\nabla^{M} f|^{2}
\end{equation}
and 
\begin{equation}
\label{Bochner-2}
-\frac{1}{2} \Delta^{F_{y}} |\nabla^{F_{y}}(f\restriction_{F_{y}})|^{2} = |\mbox{Hess}^{F_{y}}(f\restriction_{F_{y}})|^{2} -(\lambda_{k}(g)-a-c) |\nabla^{F_{y}}( f\restriction_{F_{y}})|^{2}.
\end{equation}
Since $f$ is not a pull-back of a function on $B$, one can easily obtain $\lambda_{k}(g)-a-c >0$. Thus combining $(\ref{a})$, $(\ref{Bochner-1})$ and $(\ref{Bochner-2})$, one obtains
\begin{equation*}
\begin{split}
a &\geq \frac{1}{\lambda_{k}(g)-\widetilde{c}} \int_{M} |\mbox{Hess}^{M}f|^{2}d\mu_{g} \\
& \quad  - \frac{1}{\lambda_{k}(g)-a-c} \int_{B}\left[  \int_{F_{y}} |\mbox{Hess}^{F_{y}}(f\restriction_{F_{y}})|^{2} d\mu_{\iota^{*}g}(z) \right] d\mu_{j}(y). \\
\end{split}
\end{equation*}
Hence the assumption $a \leq \widetilde{c}-c$  implies 
\begin{equation}
\label{a-bound}
\begin{split}
a &\geq \left( \frac{1}{\lambda_{k}(g)-\widetilde{c}}  - \frac{1}{\lambda_{k}(g)-a-c} \right) \int_{M} |\mbox{Hess}^{M}f|^{2}d\mu_{g} \\
& \quad  + \frac{1}{\lambda_{k}(g)-a-c} \int_{B}\left[  \int_{F_{y}}(|\mbox{Hess}^{M}f|^{2} - |\mbox{Hess}^{F_{y}}(f\restriction_{F_{y}})|^{2} )d\mu_{\iota^{*}g}(z) \right] d\mu_{j}(y). \\
\end{split}
\end{equation}
We evaluate the second term on the right hand side. Let $\{\tilde{e}_{i} \}_{i=1}^{p}$ be a local orthonormal frame around $y \in B$. For any $1 \leq i \leq p$, let $e_{i}$ be the horizontal lift of $\tilde{e}_{i}$. Let $\{e_{i}\}_{i=p+1}^{n}$ be local vertical vector fields such that $\{e_{i}\}_{i=1}^{n}$ is a local orthonormal frame on $(M,g)$. Then for any $p+1 \leq i \leq n$, one has 
\begin{equation*}
e_{i}(f\restriction_{F_{y}})_{x} =df_{\iota(x)}(d\iota_{x} (e_{i}) )= df_{\iota(x)}(e_{i})= (e_{i}f)_{\iota(x)}, 
\end{equation*}
which implies
\begin{equation*}
e_{i}(f\restriction_{F_{y}}) = (e_{i}f)\restriction_{F_{y}}. 
\end{equation*}
Using this fact and the assumption that $F_{y}$ is totally geodesic, one obtains
\begin{equation*}
\begin{split}
\mbox{Hess}^{M}f(e_{i_{1}}, e_{i_{2}})\restriction_{F_{y}} &= e_{i_{1}}(e_{i_{2}}(f\restriction_{F_{y}}))- \left((\nabla^{M}_{e_{i_{1}}}e_{i_{2}})f \right)\restriction_{F_{y}} \\
&= e_{i_{1}}(e_{i_{2}}(f\restriction_{F_{y}}))- \left((\nabla^{F_{y}}_{e_{i_{1}}}e_{i_{2}})f \right)\restriction_{F_{y}} \\
&= e_{i_{1}}(e_{i_{2}}(f\restriction_{F_{y}}))- \nabla^{F_{y}}_{e_{i_{1}}}e_{i_{2}}(f\restriction_{F_{y}}) \\
&= \mbox{Hess}^{F_{y}}(f\restriction_{F_{y}}) (e_{i_{1}}, e_{i_{2}}) \\
\end{split}
\end{equation*}
for any $p+1 \leq i_{1}, i_{2} \leq n$. Set
\begin{equation*}
|\mbox{Hess}_{h}f|^{2} := \sum_{i_{1}, i_{2}=1}^{p}\left[ e_{i_{1}}(e_{i_{2}}f)- (\nabla^{M}_{e_{i_{1}}}e_{i_{2}})f \right]^{2}.
\end{equation*}
Then using $(\ref{local-h})$ and the Cauchy--Schwarz inequality, one obtains
\begin{equation}
\label{horizontal-Hessian}
|\mbox{Hess}^{M}f|^{2} - |\mbox{Hess}^{F_{y}}(f\restriction_{F_{y}})|^{2} \geq  |\mbox{Hess}_{h}f|^{2} \geq \frac{(\Delta_{h}f)^{2}}{p} = \frac{a^{2}}{p}f^{2}
\end{equation}
on each fiber $F_{y}$. The Cauchy--Schwarz inequality implies 
\begin{equation}
\label{Hessian}
|\mbox{Hess}^{M}f|^{2} \geq \frac{(\Delta^{M}f)^{2}}{n} =\frac{ \lambda_{k}(g)^{2}}{n}f^{2}.
\end{equation}
Combining $(\ref{a-bound})$, $(\ref{horizontal-Hessian})$ and $(\ref{Hessian})$, one obtains
\begin{equation*}
\frac{a^{2}}{(\lambda_{k}(g)-a-c)p} -a  +\left( \frac{1}{\lambda_{k}(g)-\widetilde{c}}  - \frac{1}{\lambda_{k}(g)-a-c} \right) \frac{\lambda_{k}(g)^{2}}{n} \leq 0.
\end{equation*}
After multiplying both sides by $(\lambda_{k}(g)-a-c)p >0$, one obtains $Q_{k}(a) \leq 0$.

\underline{Case: $p=n-1$} Suppose $p=n-1$. Then by the assumption of Theorem \ref{MainThm}, each fiber $F_{y}$ is a closed geodesic on $(M,g)$. Hence the Bochner formula $(\ref{comparison})$ cannot be used. However the following formula, which immediately follows from the Leibniz rule, can be used alternatively: for a smooth function $f$ on $M$,
\begin{equation*}
\frac{1}{2}(((f'(s))^{2})'' = (f''(s))^{2} + f'(s)f'''(s),
 \end{equation*}
 where $s$ is the parameter of $F_{y}$ with unit speed and $f(s)$ is the restriction of $f$ on $F_{y}$. Consider the second derivative as $-\Delta^{F_{y}}$ or $\mbox{Hess}^{F_{y}}$, and the first derivative as $\nabla^{F_{y}}$. Let $c=0$ in the above discussion. Then we can prove the assertion in a similar manner and so we omit the detailed proof.
 
 Next we prove the second half of the assertion. If $a$ satisfies $a \leq \widetilde{c}-c$ and $Q_{k}(a) = 0$, then the equality holds in $(\ref{Hessian})$. Hence the equality condition of the Cauchy--Schwarz inequality implies 
 \begin{equation*}
 \mbox{Hess}^{M}f=- \frac{\lambda_{k}(g)}{n}f g.
 \end{equation*}
Thus the result due to Obata \cite{Obata}, which can be used to prove the equality condition for the Lichnerowicz--Obata theorem, implies that $(M,g)$ is isometric to the round sphere of radius $\sqrt{n/\lambda_{k}(g)}= \sqrt{(n-1)/\widetilde{c}}$, and that $f$ is a linear combination of coordinate functions and so $k=1$. The classification result due to Escobales \cite{Escobales} and Ranjan \cite{Ranjan} implies that if the round sphere $S^{n}$ admits a Riemannian submersions with totally geodesic fibers, then $n$ must be odd. 

We show the converse. Assume that $(M,g)$ is isometric to $S^{n}(\sqrt{(n-1)/\widetilde{c}})$ and $n$ is odd. Since it is classically known that a connected totally geodesic submanifold of a round sphere is a lower dimensional round sphere, each fiber $F_{y}$ of $\pi$ is isometric to $S^{n-p}(\sqrt{(n-1)/\widetilde{c}})$. In particular, one has 
\begin{equation*}
\mbox{Ric}^{F_{y}} = \frac{(n-p-1)\widetilde{c}}{n-1}g
\end{equation*}
and so $c= (n-p-1)\widetilde{c}/(n-1)$. If $f$ is a linear combination of coordinate function, then $f\restriction_{F_{y}}$ is again a linear combination of coordinate functions. Hence one has 
\begin{equation*}
\Delta^{M}f= \frac{n\widetilde{c}}{n-1}f \quad \mbox{and} \quad \Delta_{v}f = \frac{(n-p)\widetilde{c}}{n-1}f.
\end{equation*}
Thus we obtain
\begin{equation*}
\Delta_{h}f = \frac{p\widetilde{c}}{n-1}f
\end{equation*}
and so $a=p\widetilde{c}/(n-1)$. In particular, one has $a = \widetilde{c}-c$ = $p\widetilde{c}/(n-1)$. Furthermore, one has
\begin{equation*}
\begin{split}
Q_{1}(x) &= (p+1)x^{2}- \frac{p(n+p+1)\widetilde{c}}{n-1}x + \frac{np^{2}\widetilde{c}^{2}}{(n-1)^{2}} \\
&= \left(x- \frac{p\widetilde{c}}{n-1} \right) \left((p+1)x - \frac{np\widetilde{c}}{n-1} \right). \\
\end{split}
\end{equation*}
Hence one concludes $Q_{1}(a) = 0$. The proof is completed.
\end{proof}
In the second half of the above proof, we have used the classification result due to  Escobales \cite{Escobales} and Ranjan \cite{Ranjan}. See Section 4 for the classification result.

Using the above lemma, we prove the following lemma:

\begin{lemm}
\label{lemma}
\textit{
Let $(M, g)$ and $(B, j)$ be compact Riemannian manifolds of dimension $n$ and $p$ respectively. Let $\pi: (M, g) \to (B,j)$ be a Riemannian submersion. Assume that there exists $\widetilde{c} > 0$ such that 
\begin{equation*}
\mbox{Ric}^{M} \geq \widetilde{c}g,
\end{equation*}
where $\mbox{Ric}^{M}$ is the Ricci tensor of $(M,g)$. In case $p \leq n-2$, assume also that there exists $0 \leq c < \widetilde{c}$ such that 
\begin{equation*}
\mbox{Ric}^{F_{y}} = c(\iota^{*}g)
\end{equation*}
for any $y \in B$, where $F_{y}$ is a fiber of $y$ and $\iota: F_{y} \hookrightarrow M$ is the inclusion map. Let $f$ be a smooth function on $M$ such that $\int_{M}f^{2} d\mu_{g} =1$, $\Delta^{M}f= \lambda_{k}(g)f$, $\Delta_{h}f=af$. Then one has 
\begin{equation}
\label{a-lower-bound}
a > \frac{\widetilde{c}-c}{n+1},
\end{equation}
where $c$ is regarded as 0 in case $p=n-1$. }
\end{lemm}

\begin{proof}
If one has $a > \widetilde{c}-c$, then there is nothing to prove. Hence we assume $a \leq \widetilde{c}-c$. Then Lemma \ref{key-lemma} implies $Q_{k}(a) \leq 0$. Since $\Delta_{h}$ has nonnegative eigenvalues on the space of $\lambda_{k}(g)$-eigenfunctions, the quadratic function $Q_{k}$ has one or two roots. First we show that $(\widetilde{c}-c)/(n+1)$ sits on the left hand side of the axis of symmetry of $Q_{k}$. The axis of symmetry is given by

\begin{equation*}
 x= \frac{\alpha_{k}}{2(p+1)} = \frac{p}{2(p+1)} \left[\lambda_{k}(g)-c+\frac{\lambda_{k}(g)^{2}}{n(\lambda_{k}(g)-\widetilde{c})} \right].
\end{equation*}
Hence the fact that $p/2(p+1)$ is the smallest for $p=1$ and the Lichnerowicz--Obata theorem imply
\begin{equation*}
\begin{split}
\frac{\alpha_{k}}{2(p+1)} &> \frac{1}{4}\left[\lambda_{1}(g) -c+\frac{\lambda_{1}(g)}{n} \right] \\
& \geq \frac{1}{4} \left( \frac{n+1}{n-1}\widetilde{c}-c \right) \\
& > \frac{\widetilde{c}-c }{4}. \\
\end{split}
\end{equation*}
As we have remarked at the beginning of this section, we have $n \geq 3$. Hence we have
\begin{equation*}
\frac{\alpha_{k}}{2(p+1)} > \frac{\widetilde{c}-c }{4} \geq \frac{\widetilde{c}-c}{n+1}.
\end{equation*}
Thus $(\widetilde{c}-c)/(n+1)$ sits on the left hand side of the axis of symmetry of $Q_{k}$.

For simplicity, we consider the positive real number $l$ satisfying $n^{-l} = (\widetilde{c}-c)/(n+1)$. Then we have
\begin{equation*}
Q_{k}(n^{-l}) =  (p+1) n^{-2l}+pcn^{-l}  + \frac{\lambda_{k}(g)p}{n^{l+1}(\lambda_{k}(g)-\widetilde{c})}\left[ n\widetilde{c}+\lambda_{k}(g) \{ n^{l}(\widetilde{c}-c)-n-1 \} \right]. 
\end{equation*}
Hence by the definition of $l$, $Q_{k}(n^{-l})$ is positive. Thus we conclude $(\ref{a-lower-bound})$.
\end{proof}

From the second half of Lemma \ref{key-lemma}, it seems difficult to obtain a sharp lower bound for the eigenvalue of the horizontal Laplacian. We prove Theorem \ref{MainThm}.

\begin{proof}[Proof of Theorem \ref{MainThm}]
Since $\Delta^{M}$ and $\Delta_{h}$ commute, there exists a complete orthonormal system of eigenfunctions $\{\varphi_{k}^{(i)} \}_{k \in \mathbf{N}\cup \{0\}, 1\leq i \leq m(k)}$ in $L^{2}(M,g)$ such that $\Delta^{M}\varphi_{k}^{(i)} = \lambda_{k}(g) \varphi_{k}^{(i)}$, $\Delta_{h}\varphi_{k}^{(i)} = a_{k,i}\varphi_{k}^{(i)}$ and  $a_{k,1} \leq a_{k,2} \leq \cdots \leq a_{k, m(i)}$, where $m(i)$ is the multiplicity of $\lambda_{k}(g)$. Then $(\ref{variation-Laplacian})$ implies
\begin{equation*}
\Delta^{M}_{g_{t}} \varphi_{k}^{(i)} =  \left[ t^{-2}\lambda_{k}(g) +(1-t^{-2})a_{k,i}  \right]\varphi_{k}^{(i)}
 \end{equation*}
 and so one has 
 \begin{equation*}
\langle \Delta^{M}_{g_{t}} \varphi_{k}^{(i)}, \varphi_{k}^{(i)} \rangle_{L^{2}(g)}  =   \left[ t^{-2}\lambda_{k}(g) +(1-t^{-2})a_{k,i}  \right].
\end{equation*}
Hence one obtains
\begin{equation*}
\lambda_{1}(g_{t}) = \inf_{k \in \mathbf{N}} \left[ t^{-2}\lambda_{k}(g) +(1-t^{-2})a_{k,1}  \right].
\end{equation*}
Hence for any $t \geq 1$, the Lichnerowicz--Obata theorem and the above lemma imply
\begin{equation}
\begin{split}
\lambda_{1}(g_{t}) &\geq t^{-2} \lambda_{1}(g_{t}) + (1-t^{-2})\cdot \frac{\widetilde{c}-c}{n+1} \\
&\geq \frac{\widetilde{c}-c}{n+1} + t^{-2}\left( \frac{n^{2}+1}{n^{2}-1}\widetilde{c}+\frac{c}{n+1} \right). \\
\end{split}
\end{equation}
If the equality holds in the above inequality, then one must have $t=1$ since $(\ref{a-lower-bound})$ is a strict inequality. Furthermore, the Lichnerowicz--Obata theorem implies that $(M,g)$ is isometric to the round sphere $S^{n}(\sqrt{(n-1)/\widetilde{c}})$. The classification result due to Escobales \cite{Escobales} and Ranjan \cite{Ranjan} implies that $n$ is odd. Conversely, if $t=1$ and $(M,g)$ is isometric to the odd dimensional round sphere $S^{2m+1}(\sqrt{2m/\widetilde{c}})$, then the equality clearly holds. 

\end{proof}

Also, since we have $\mbox{Vol}(M,g_{t}) = t^{n-p}\mbox{Vol}(M,g)$, we conclude Corollary \ref{Intro-corr}.

\section{Stability in the Yamabe Problem}
In this section, we apply Theorem \ref{MainThm} to the stability problem of critical points of the Yamabe functional. Before stating the application, we briefly review basic facts about the Yamabe problem. (For details, see \cite{Aubin2} and \cite{LPZ} for example.) Let $M$ be a compact (connected) manifold of dimension $n$ (without boundary). For a conformal class $C$ on $M$, the Yamabe functional $Y_{C}$ is defined as the normalized Einstein--Hilbert functional restricted to $C$, namely:
\begin{equation*}
Y_{C}: C \to \mathbf{R}, \quad Y_{C}(g) = \mbox{Vol}(M,g)^{(2-n)/n} \int_{M} S(g) \: d\mu_{g},
\end{equation*}
where $S(g)$ is the scalar curvature defined by $g$. The Yamabe problem is the minimization problem of $Y_{C}$. By virtue of the combined efforts of Yamabe \cite{Yamabe}, Trudinger \cite{Trudinger}, Aubin \cite{Aubin} and Schoen \cite{Schoen}, it is now known that there exists a minimizer of the Yamabe functional $Y_{C}$ for any conformal class $C$ on any compact manifold. A metric $g \in C$ is a critical point of $Y_{C}$ if and only if $S(g)$ is constant on $M$. Let $g \in C$ be a critical point of $Y_{C}$ and  $L^{2}_{0}(M,g)$ the space of functions in $L^{2}(M,g)$ whose integral over $(M,g)$ is 0. The \textit{Jacobi operator} $J_{g}: L^{2}_{0}(M,g) \to L^{2}_{0}(M,g)$ is defined by
\begin{equation*}
J_{g} := \Delta_{g} -\frac{S(g)}{n-1} \cdot \mbox{Id}.
\end{equation*}
The critical point $g$ is called \textit{nondegenerate} if $\mbox{Ker} J_{g} = \{0\}$ and \textit{degenerate} otherwise. The second variation of $Y_{C}$ at the critical point $g \in C$ is positive if and only if $J_{g}$ is positive-definite. In this case, $g$ is nondegenerate and a strict local minimum for $Y_{C}$, and called \textit{stable}. Following \cite{BLP}, we say that $g$ is \textit{degenerate stable} if $J_{g}$ is positive-semidefinite with nontrivial kernel.

Let $(M, g)$ and $(B, j)$ be compact Riemannian manifolds of dimension $n$ and $p$ $(n>p)$ respectively. Let $\pi: (M, g) \to (B,j)$ be a Riemannian submersion. We recall that $g$ is called \textit{Einstein} if there exists a constant $\mu_{g}$ such that $\mbox{Ric}^{M} = \mu_{g}g$, where $\mbox{Ric}^{M}$ is the Ricci tensor of $(M,g)$. In Theorem \ref{MainThm}, we assume that $\mbox{Ric}^{M}$ satisfies $\mbox{Ric}^{M} \geq \widetilde{c}g$ for some $\widetilde{c} > 0$. In this section, we always assume that $\mbox{Ric}^{M}= \widetilde{c}g$ for some $\widetilde{c} > 0$ and so that $g$ is Einstein. We recall that for $\pi: (M,g) \to (B,j)$, the $(2,1)$ tensor field $A$ on $M$ is defined 
\begin{equation*}
A_{E_{1}}E_{2} = \mathcal{H}\nabla_{\mathcal{H}E_{1}}\mathcal{V}E_{2} + \mathcal{V}\nabla_{\mathcal{H}E_{1}}\mathcal{H}E_{2},
\end{equation*}
where $\mathcal{H}$ and $\mathcal{V}$ denote the horizontal and vertical projection respectively. The assumption that $g$ is Einstein implies that the scalar curvature of each fiber is constant, that $|A|^{2}$ is constant on $M$ and that the scalar curvature $S^{B}$ of $(B,j)$ is constant (see \cite[p.250]{Besse}). If $A=0$, then $(M,g)$ is locally isometric to the Riemannian product $(B \times F, j \oplus \iota^{*}g)$. We assume $|A|>0$. As for the canonical variation, the O'Neill formula (\cite[p. 253]{Besse}) reads
\begin{equation}
\label{ONeill}
S(g_{t}) = -t^{2}|A|^{2} + S^{B} +t^{-2}S^{F},
\end{equation}
where $S^{F}$ is the scalar curvature of each fiber $(F, \iota^{*}g)$. This formula in particular implies that if $g$ is Einstein, then for any $t>0$, $S(g_{t})$ is constant on $M$ and so $g_{t}$ is a critical point of $Y_{[g]}$. Furthermore, if $|A|>0$, then $g_{t}$ is stable for sufficiently large $t$. Using Theorem \ref{MainThm}, we obtain the following quantitative estimate:
\begin{theo}
\label{stability}
\textit{
Let $(M, g)$ and $(B, j)$ be connected compact Riemannian manifolds of dimension $n$ and $p$ $(n>p)$ respectively. Assume that $(M,g)$ is not isometric to the round sphere. Let $\pi: (M, g) \to (B,j)$ be a Riemannian submersion each of whose fibers is connected and totally geodesic. Assume $|A|>0$. Assume also that there exists $\widetilde{c} > 0$ such that 
\begin{equation*}
\mbox{Ric}^{M} = \widetilde{c}g,
\end{equation*}
where $\mbox{Ric}^{M}$ is the Ricci tensor of $(M,g)$. In case $p \leq n-2$, assume also that there exists $0 \leq c < \widetilde{c}$ such that 
\begin{equation*}
\mbox{Ric}^{F_{y}} = c(\iota^{*}g)
\end{equation*}
for any $y \in B$. Set
\begin{equation*}
\Gamma := \frac{n^{2}+1}{n+1}\left( \widetilde{c}-c \right) +pc.
\end{equation*}
Then for any $t \geq \max\{1, \: \sqrt{\Gamma}/|A|\}$, $g_{t}$ is stable. }
\end{theo}

Before the proof, we remark that study of stability of the canonical variation of the round metric on a sphere is completed by Bettiol--Lauret--Piccione \cite{BLP} (see the list on page 61 of \cite{BLP}) .

\begin{proof}

By the assumption, $(\ref{ONeill})$ can be rewritten as

\begin{equation*}
S(g_{t}) = -t^{2}|A|^{2} +S^{B} + t^{-2}(n-p)c.
\end{equation*}
Hence in particular, one obtains
\begin{equation*}
n \widetilde{c} =S(g) = -|A|^{2}+ S^{B} + (n-p)c .
\end{equation*}
Hence we have 
\begin{equation}
\label{Sgt}
S(g_{t}) = -t^{2}|A|^{2}+ |A|^{2}-(n-p)c+n\widetilde{c}+ t^{-2}(n-p)c. 
\end{equation}
Combining $(\ref{Sgt})$ with Theorem \ref{MainThm}, one obtains
\begin{equation*}
\begin{split}
(n-1) \left( \lambda_{1}(g_{t}) - S(g_{t}) \right) &> |A|^{2}t^{2}-\left( \Gamma + |A|^{2} \right) + \Gamma t^{-2} \\
&= t^{-2} \left( |A|^{2}t^{4}-\left( \Gamma + |A|^{2} \right) t^{2} + \Gamma \right) \\
&= |A|^{2}t^{-2} \left( t^{2}-\frac{\Gamma}{|A|^{2}} \right) \left( t^{2}-1 \right)  \\
\end{split}
\end{equation*}
for any $t \geq 1$. Thus we conclude the assertion.

\end{proof}

\section{Examples}
Before seeing examples of application of Theorem \ref{MainThm}, we will see examples that do not satisfy the curvature assumption of Theorem \ref{MainThm}. 
\begin{exam}\label{torus} We consider the standard torus $T^{n} := \mathbf{R}^{n}/\mathbf{Z}^{n}$ $(n \geq 2)$ equipped with the standard flat metric $g = \langle \cdot, \cdot \rangle$. Let $(x^{1}, \ldots, x^{n})$ be the standard coordinate on $T^{n}$ and consider the standard projection 
\begin{equation*}
\pi: T^{n} \rightarrow T^{n-1}, \quad (x^{1}, \ldots, x^{n}) \mapsto (x^{1}, \ldots, x^{n-1}).
\end{equation*}
Clearly $\pi$ is a Riemannian submersion each of whose fiber is connected and totally geodesic. For each $1 \leq i \leq n$, set $\partial_{i} := \frac{\partial}{\partial x^{i}}$. Then one has 
\begin{equation*}
\Delta^{T^{n}} = - \sum_{i=1}^{n} \partial_{i}^{2}, \quad \Delta_{v} = -\partial_{n}^{2}, \quad \Delta_{h}= - \sum_{i=1}^{n-1} \partial_{i}^{2}.
\end{equation*}
It is known that 
\begin{equation*}
\mbox{Spec}(\Delta^{T^{n}}) = \{ 4\pi^{2}|y|^{2} \mid y \in \mathbf{Z}^{n} \}
\end{equation*} 
holds.  Moreover, for each $y, -y \in \mathbf{Z}^{n}$, the space of the corresponding eigenfunctions is $\mbox{span}\{ \varphi_{y}(x):= \mbox{cos}(2\pi \langle x, y \rangle ), \psi_{y}(x):= \mbox{sin}(2\pi \langle x, y \rangle ) \}$ (see \cite[pp.272-273]{Sakai}, for instance). Let $\{e_{i} \}_{i=1}^{n}$ be the standard orthonormal basis of $\mathbf{R}^{n}$. By $(\ref{variation-Laplacian})$, one can see that for $t >1$, $\lambda_{1}(g_{t}) = 4 \pi^{2}t^{-2}$ holds and the space of the corresponding eigenfunctions is $\mbox{span}\{ \varphi_{e_{n}}(x)= \mbox{cos}(2\pi x^{n}), \psi_{e_{n}}(x) = \mbox{sin}(2\pi x^{n}) \}$. Thus one concludes that
\begin{equation*}
\Lambda_{1}(T^{n},t) = \lambda_{1}(g_{t}) \mbox{Vol}(T^{n}, g_{t})^{2/n} = 4\pi^{2}t^{2(1-n)/n} \rightarrow 0 \quad (t \rightarrow \infty).
\end{equation*}
\end{exam}
This example shows that the assumption $c < \widetilde{c}$ in Theorem \ref{MainThm} and Corollary $\ref{Intro-corr}$ cannot be weakened to $c \leq \widetilde{c}$. The above example can be generalized to the following example:

\begin{exam}
Let $(B.j)$ and $(F,h)$ be connected compact Riemannian manifolds of dimension $p$ and $q$ respectively. We assume that there exists $c \geq 0$ such that $\mbox{Ric}^{F} = c h$. We consider the Riemannian product $(B \times F, g= j \oplus h)$ and the canonical submersion
\begin{equation*}
\pi: B \times F \to B, (x,y) \mapsto x.
\end{equation*}
Each fiber is $(F,h)$ and totally geodesic. Since the Ricci tensor on $(B \times F, j \oplus h)$ is given by
\begin{equation*}
\mbox{Ric}^{B \times F} = \left(\mbox{Ric}^{B} \oplus 0 \right) + \left( 0 \oplus \mbox{Ric}^{F} \right),
\end{equation*}
$\widetilde{c}$ satisfying $\mbox{Ric}^{B \times F} \geq \widetilde{c} (j \oplus h)$ is at most $c$. In case $q=1$, let $c=0$ in the above. Hence this submersion $\pi$ does not satisfy the curvature assumption of Theorem \ref{MainThm}. In fact, one can prove the following:
\begin{prop}
\label{product-prop}
\textit{
In the above setting, one has 
\begin{equation*}
\lim_{t \to \infty} \Lambda_{1}(B \times F, t) = 0
\end{equation*} 
with respect to the canonical variation $(g_{t} = j \oplus t^{2}h)_{t>0}$. }
\end{prop}
This proposition shows that one cannot necessarily conclude $\lim_{t \to \infty} \Lambda_{1}(M,t) = \infty$ if the assumption $c < \widetilde{c}$ in Theorem \ref{MainThm} is weakened to $c \leq \widetilde{c}$. The proof of this proposition is generalization of the discussion in Example \ref{torus}. Before the proof, we recall basic facts about the spectrum of the Laplacian on a product of Riemannian manifolds. For functions $\varphi \in C^{\infty}(B)$ and $\psi \in C^{\infty}(F)$, we define the function $\varphi \times \psi$ on $B \times F$ by
\begin{equation*}
(\varphi \times \psi)(x, y) := \varphi (x)\psi (y), \quad (x,y) \in B \times F.
\end{equation*}
The set $\{ \varphi \times \psi \mid \varphi \in C^{\infty}(B), \psi \in C^{\infty}(F) \}$ is dense in $L^{2}(B \times F)$. We have $\Delta^{B\times F} (\varphi \times \psi) = (\Delta^{B}\varphi) \times \psi + \varphi + (\Delta^{F}\psi)$. Each eigenvalue of $\Delta^{B\times F}$ is of the form $\lambda_{k}(B,j) + \lambda_{l}(F,h)$ and its corresponding eigenfunction is of the form $\varphi\times \psi$ with $\Delta^{B} \varphi  = \lambda_{k}(B,j) \varphi$ and $\Delta^{F} \psi  = \lambda_{l}(F,h) \psi$. 

\begin{proof}[Proof of Proposition \ref{product-prop}]
Let $T>0$ be the number satisfying $T^{-2} \lambda_{1}(F,h) = \lambda_{1}(B,j)$. Let $\psi \in C^{\infty}(F)$ be a function satisfying $\Delta^{F} \psi = \lambda_{1}(F,h)\psi$. Then one clearly has
\begin{equation*}
\Delta_{v} (1\times \psi) = 1\times( \Delta^{F}\psi) = \lambda_{1}(F,h) (1\times \psi)
\end{equation*}
and so $\Delta_{h} (1\times \psi) =0$. Hence, using $(\ref{variation-Laplacian})$, one can easily see that for $t>T$, $\lambda_{1}(B \times F, g_{t}) = \lambda_{1}(F,h) t^{-2}$ holds and the space of the corresponding eigenfunctions is $\mbox{span}\{ 1 \times \psi \mid  \psi \in C^{\infty}(F), \Delta^{F} \psi = \lambda_{1}(F,h)\psi \}$. Thus one concludes
\begin{equation*}
\Lambda_{1}(B \times F,t)  =  \lambda_{1}(F,h)\mbox{Vol}(B \times F, g) t^{2/(p+q)-2}  \to 0 \quad  (t \to \infty).
\end{equation*} 
\end{proof}

\end{exam}
It is well known that compact rank one symmetric spaces are exactly $S^{n}$, $\mathbf{R}P^{n}$, $\mathbf{C}P^{n}$, $\mathbf{H}P^{n}$ and the Cayley projective plane $\mathbf{Ca}P^{2}$. In what follows, we always assume that $\mathbf{C}P^{n}$, $\mathbf{H}P^{n}$ and $\mathbf{Ca}P^{2}$ are endowed with the canonical metrics such that the sectional curvatures on them lie in the iterval $[1,4]$. First we consider Riemannian submersions from round spheres, all of which satisfy the assumption of Theorem \ref{MainThm}. 
\begin{exam}
\label{exam-sphere}
We consider the following Riemannian submersions:
\begin{enumerate}[(1)]
\item $S^{1} \to S^{2n+1} \to \mathbf{C}P^{n}$,
\item $S^{3} \to S^{4n+3} \to \mathbf{H}P^{n}$,
\item $S^{7} \to S^{15} \to S^{8}(1/2)$.
\end{enumerate}

First we consider the submersion (1). Since the first eigenvalue of the round sphere $S^{m}$ is $m$ and that of $\mathbf{C}P^{n}$ is $4(n+1)$, Theorem \ref{BBB-theo} implies that the inequality
\begin{equation*}
2n+1 \leq \lambda_{1}(g_{t}) \leq 4(n+1)
\end{equation*}
holds for any $0<t\leq 1$. Moreover, this submersion satisfies the assumption of Theorem \ref{MainThm} and so Theorem \ref{MainThm} implies that the inequality
\begin{equation*}
\frac{n}{n+1} + \frac{2n^{2}+2n+1}{n+1} t^{-2} \leq \lambda_{1}(g_{t}) \leq 4(n+1)
\end{equation*}
holds for any $1 \leq t$. In fact, Tanno \cite{Tanno} obtained 
\begin{equation*}
\lambda_{1}(S^{2n+1}, g_{t}) = \min\{2n+t^{-2}, 4(n+1)\},
\end{equation*}
which satisfies the above two inequalities. By combining this result with the classical fact that the scalar curvature $S(g_{t})$ of $(S^{2n+1}, g_{t})$ is given by $S(g_{t}) = 2n(2n+2-t^{2})$, one can easily show that $g_{t}$ is stable for any $t \neq 1$ and degenerate stable for $t=1$. Tanno himself did not relate his work \cite{Tanno} with the Yamabe problem and this stability condition was first pointed out by Bettiol and Piccione \cite{BP}. (See \cite{BP} for the detailed proof.) The stability condition that can be obtained from Theorem \ref{stability} is not optimal in this case. 

Similarly, we consider the submersion (2). Since the first eigenvalue of $\mathbf{H}P^{n}$ is $8(n+1)$, Theorem \ref{BBB-theo} implies that the inequality
\begin{equation*}
4n+3 \leq \lambda_{1}(g_{t}) \leq 8(n+1)
\end{equation*}
holds for any $0<t\leq 1$. Moreover, Theorem \ref{MainThm} implies that the inequality
\begin{equation*}
\frac{n}{n+1}+ \frac{4n^{2}+6n+3}{n+1}t^{-2} \leq \lambda_{1}(g_{t}) \leq 8(n+1)
\end{equation*}
holds for any $1 \leq t$. In fact, Tanno \cite{Tanno-Tsukuba} obtained 
\begin{equation}
\label{Tanno-4n+3}
\lambda_{1}(S^{4n+3}, g_{t}) = \min\{4n+3t^{-2}, 8(n+1)\},
\end{equation}
which satisfies the above two inequalities. By the O'Neill formula $(\ref{ONeill})$, one obtains $S(g_{t}) = -12nt^{2}+16n(n+2)+6t^{-2}$. Based on these facts, Bettiol and Piccione \cite{BP} proved that $g_{t}$ is stable if and only if $6nt^{4}+8(n^{2}+n+1)t^{2}-3>0$ and $t\neq 1$. (See also \cite{BLP}.)

With respect to (3), the first eigenvalue of $S^{8}(1/2)$ is $32$ and so Theorem \ref{BBB-theo} implies that the inequality
\begin{equation*}
15 \leq \lambda_{1}(g_{t}) \leq 32
\end{equation*}
holds for any $0<t\leq 1$. Moreover, this submersion satisfies the assumption of Theorem \ref{MainThm} and so Theorem \ref{MainThm} implies that the inequality
\begin{equation*}
\frac{1}{2}+ \frac{29}{2}t^{-2} \leq \lambda_{1}(g_{t}) \leq 32
\end{equation*}
holds for any $1\leq t$. In fact, Bettiol and Piccione \cite{BP} obtained 
\begin{equation*}
\lambda_{1}(S^{15}, g_{t} ) =  \min\{8+7t^{-2}, 32\},
\end{equation*}
 which satisfies the above two inequalities. In \cite{BLP}, it was shown that $g_{t}$ is stable if and only if $t>\sqrt{(\sqrt{19}-4)/2} \approx 0.4236$ and $t \neq 1$.

\begin{rema}
Escobales \cite{Escobales} and Ranjan \cite{Ranjan} completed the classification of Riemannian submersions from round spheres with totally geodesic fibers and showed that such submersions are exactly the above ones. In fact, by the subsequent work of Gromoll and Grove \cite{GG} and that of Wilking \cite{Wilking}, any Riemannian submersion with connected fibers from a round sphere is one of the above (see also \cite[Theorem 4.3.3]{GW}).
\end{rema}

\end{exam}

Before moving on to the next example, we review the following remarkable result due to Escobales \cite{Escobales-Tokyo}: 
\begin{theo}[\cite{Escobales-Tokyo}]
\textit{
Let $(M,g)$ be one of $\mathbf{C}P^{m}$, $\mathbf{H}P^{m}$ and $\mathbf{Ca}P^{2}$. Assume that $\pi: (M,g) \to (B,j)$ is a Riemannian submersion with connected oriented fibers onto an oriented compact manifold $B$. Then $(M,g)$ is $\mathbf{C}P^{2n+1}$ and $\pi$ is nothing but the canonical fibration $\mathbf{C}P^{1} \to \mathbf{C}P^{2n+1} \to \mathbf{H}P^{n}$. In particular, each fiber of $\pi$ is totally geodesic. }
\end{theo}
\begin{exam}
\label{exam-cp}
Motivated by the above theorem, we consider the fibration $\mathbf{C}P^{1} \to \mathbf{C}P^{2n+1} \to \mathbf{H}P^{n}$. Since the first eigenvalue of $\mathbf{C}P^{2n+1}$ is $8(n+1)$ and that of $\mathbf{H}P^{n}$ is also $8(n+1)$, Theorem \ref{BBB-theo} implies that the equation
\begin{equation}
\label{cp-1}
\lambda_{1}(g_{t}) =8(n+1)
\end{equation}
holds for any $0<t\leq 1$. Moreover, Theorem \ref{MainThm} implies that the inequality 
\begin{equation}
\label{cp-2}
\frac{4n}{4n+3}+\frac{64n^{3}+128n^{2}+100n+24}{16n^{2}+16n+3}t^{-2} < \lambda_{1}(g_{t}) \leq 8(n+1)
\end{equation}
holds for any $1 \leq t$. In fact, Bettiol, Lauret and Piccione \cite{BLP} recently proved 
\begin{equation*}
\lambda_{1}(\mathbf{C}P^{2n+1}, g_{t}) = \min\{8n+ 8t^{-2}, 8(n+1) \},
\end{equation*}
which satisfies $(\ref{cp-1})$ and $(\ref{cp-2})$. In \cite{BLP}, it was shown that $g_{t}$ is stable if and only if $\displaystyle t > \sqrt{\frac{\sqrt{(2n^{2}+n+1)^{2}+4n}-(2n^{2}+n+1)}{2n}}$.
\end{exam}

Next we consider an example of twistor fibrations:
\begin{exam} We consider the complex flag manifold $M= F(1,2)$. It is known that there exists a K\"{a}hler--Einstein metric $g$ on $M$ such that $\mbox{Ric}^{M}=2g$ and $M$ admits a Riemannian submersion to $\mathbf{C}P^{2}$ whose fibers are totally geodesic 2-spheres with constant sectional curvature $1$. (See \cite{DM} and \cite{FK} for details.) Since Theorem $\ref{MainThm}$ implies that the inequality
\begin{equation*}
\lambda_{1}(g_{t}) \geq \frac{2}{7} + \frac{79}{35}t^{-2}
\end{equation*}
holds for any $t \geq 1$. One has $|A|^{2}=2$ and $\Gamma = 65/7$ and so Theorem \ref{stability} implies that $g_{t}$ is stable for any $t \geq \sqrt{65/14} \approx 2.1547$. The author believes that this result is new.

\end{exam}

Next we consider fibrations that are generalization of some of the above examples.

\begin{exam}
\label{Kobayashi-example}
As generalization of the Hopf fibration $S^{1} \to S^{2n+1} \to \mathbf{C}P^{n}$, Kobayashi \cite{Kobayashi} showed that over a compact K\"{a}hler--Einstein manifold $(B,j)$ with $\mbox{Ric}^{B} = 2(n+1)j$, one can construct an $S^{1}$-fibration $S^{1} \to (M,g) \to (B,j)$ whose total space $(M,g)$ is a Sasaki--Einstein manifold with $\mbox{Ric}^{M} = 2ng$. (See also \cite[pp. 84-85]{Baum}.) Theorem $\ref{MainThm}$ implies that the inequality
\begin{equation*}
\frac{n}{n+1}+ \frac{2n^{2}+2n+1}{n+1}t^{-2} \leq \lambda_{1}(g_{t})  \leq \beta_{1}
\end{equation*}
holds for any $1 \leq t$. One has $|A|^{2}=2n$ and $\displaystyle \Gamma= \frac{2n(2n^{2}+2n+1)}{n+1}$. Hence Theorem \ref{stability} implies that $g_{t}$ is stable for any $\displaystyle t \geq \sqrt{2n+\frac{1}{n+1}}$. As far as the author knows, this result is new.
\end{exam}

\begin{rema} 
Wang and Ziller \cite{WZ} obtained an extension of the aforementioned Kobayashi's result. That is, they proved that under some assumptions, there exists a Riemannian submersion from a compact Einstein manifold of positive scalar curvature onto a product of K\"{a}hler--Einstein manifolds each of whose fibers is a totally geodesic flat torus. Some other extensions are also known. (See \cite{CS} and \cite{Sakane} for example.) Theorem \ref{MainThm} and Theorem \ref{stability} can be applied to these cases too, but we will omit the details.
\end{rema}

\begin{exam} We consider generalization of $S^{3} \to S^{4n+3} \to \mathbf{H}P^{n}$, which is Example \ref{exam-sphere} (2). Let $(B, j)$ be a quaternionic K\"{a}hler manifold of dimension $4n$ $(n\geq2)$. As is well known, $(B,j)$ is Einstein. It is also known that if the Ricci curvature of $(B,j)$ is positive, then $(B,j)$ must be compact. It follows from Konishi's work  \cite{Konishi} that for a quaternionic K\"{a}hler manifold $(B,j)$ of $\mbox{Ric}^{B} = 4n+8$, there exists a Riemannian submersion from a $3$-Sasakian manifold $(M,g)$ with $\mbox{Ric}^{M} = 4n+2$ onto $(B,j)$ whose fiber is totally geodesic and of constant sectional curvature $1$. Theorem \ref{BBB-theo} and Theorem \ref{MainThm} can be applied in this case as well. However, Nagy and Semmelmann \cite{NS} recently proved that if $g$ does not have constant sectional curvature, $\lambda_{1}(g_{t}) \geq 8(n+t^{-2})$ holds for any $t>0$. This estimate is better than the estimate that can be obtained from Theorem \ref{MainThm}. As is pointed out in \cite{NS}, one has $S(g_{t}) = -12nt^{2}+16n(n+2)+6t^{-2}$. Since one has
\begin{equation*}
(4n+2)\lambda_{1}(g_{t}) - S(g_{t}) = 12nt^{2}+16n(n-1)+(32n+10)t^{-2}>0,
\end{equation*}
$g_{t}$ is stable for any $t>0$.

\end{exam}

\begin{exam} The fibration $\mathbf{C}P^{1} \to \mathbf{C}P^{2n+1} \to \mathbf{H}P^{n}$ (Example $\ref{exam-cp}$) is a typical example of the twistor fibration of a quaternionic K\"{a}hler manifold. See \cite[Chapter 14]{Besse} for details. Let $(B, j)$ be a quaternionic K\"{a}hler manifold of dimension $4n$ $(n\geq2)$. Salamon \cite{Salamon} constructed a fibration $\pi : Z \to B$ such that $Z$ is a complex manifold and each fiber of $\pi$ is biholomorphic to $\mathbf{C}P^{1}$. This fibration is called the \textit{twistor fibration} and $Z$ the \textit{twistor space}. The following theorem is known:
\begin{theo}[See {\cite[14.80]{Besse}} and its proof as well as \cite{Salamon}]
\textit{
Let $(B, j)$ be a quaternionic K\"{a}hler manifold of dimension $4n$ $(n\geq2)$ with $\mbox{Ric}^{B} = 4(n+2)j$. Then there exists a K\"{a}hler--Einstein metric $g$ on the twistor space $Z$ such that $\mbox{Ric}^{Z} = 4(n+1)g$, and $\pi : (Z,g) \to (B,j)$ is a Riemannian submersion each of whose fiber is isometric to $S^{2}(1/2)$, the round sphere of radius $1/2$, and is totally geodesic in $(Z,g)$.}
\end{theo}
This Riemannian submersion $\pi : (Z,g) \to (B,j)$ satisfies the assumption of Theorem \ref{MainThm}. Hence Theorem \ref{MainThm} implies that the inequality 
\begin{equation*}
\lambda_{1}(g_{t}) > \frac{4n}{4n+3}+\frac{64n^{3}+128n^{2}+100n+24}{16n^{2}+16n+3}t^{-2} 
\end{equation*}
holds for any $t \geq 1$. One has $|A|^{2} = 8n$ and $\displaystyle \Gamma = \frac{4n(16n^{2}+32n+17)}{4n+3}$. Hence Theorem \ref{stability} implies that $g_{t}$ is stable for $t \geq \sqrt{2n+\frac{5}{2}+\frac{1}{4n+3} }$. The author believes that this result is new.
\end{exam}

\vspace{0.1in}
 
\textbf{Acknowledgements} I would like to express my gratitude to Professor Yusuke Sakane for telling me several examples of Riemannian submersions with totally geodesic fibers. I would like to thank Professor Shin Nayatani and Professor Tatsuya Tate for their constant encouragement and valuable comments. Finally, I extend my sincere gratitude to the anonymous referee for the careful proofreading and helpful comments, which have greatly improved the presentation of this work. This research is partially supported by the Grant-in-Aid for JSPS Fellows Grant Number JP23KJ1074.

\begin{itemize}
\item Funding This work is partially supported by the Grant-in-Aid for JSPS Fellows Grant Number JP23KJ1074.
\item Competing interests The author declares no competing interests.
\item Ethics approval and consent to participate `Not applicable'
\item Consent for publication `Not applicable'
\item Data availability `Not applicable'
\item Materials availability `Not applicable'
\item Code availability `Not applicable'
\item Author contribution The author declares sole responsibility for this study.
\end{itemize}

\end{document}